\documentclass[12pt]{article}
\usepackage{amsfonts,amssymb}

\newtheorem{thm}{Theorem}[section]
\newtheorem{prop}[thm]{Proposition}
\newtheorem{lemma}[thm]{Lemma}
\newtheorem{cor}[thm]{Corollary}

\newtheorem{unnumber}{}

\newtheorem{prepf}[unnumber]{Proof}
\newenvironment{proof}{\prepf\rm}{\endprepf}
\newcommand{\qed}{\qquad$\Box$}

\newcommand{\Inn}{\mathop{\mathrm{Inn}}}

\newcommand{\Aut}{\mathop{\mathrm{Aut}}}
\newcommand{\Sym}{\mathop{\mathrm{Sym}}}

\begin{document}

\title{Pre-primitive permutation groups}
\author{Marina Anagnostopoulou-Merkouri\footnote{School of Mathematics, University of Bristol, BS8 1UG, UK; 
\texttt{marina.anagnostopoulou-merkouri@bristol.ac.uk}},\,
Peter J. Cameron\footnote{School of Mathematics and Statistics, University of St Andrews, Fife, KY16 9SS, UK; \texttt{pjc20@st-andrews.ac.uk}}\\
and Enoch Suleiman\footnote{Department of Mathematics, Federal University Gashua, Yobe State, Nigeria;
\texttt{enochsuleiman@gmail.com}}}
\date{}
\maketitle

\begin{abstract}
A transitive permutation group $G$ on a finite set $\Omega$ is said to be
\emph{pre-primitive} if every $G$-invariant partition of $\Omega$ is the
orbit partition of a subgroup of $G$. It follows that pre-primitivity and
quasiprimitivity are logically independent (there are groups satisfying
one but not the other) and their conjunction is equivalent to primitivity.
Indeed, part of the motivation for studying pre-primitivity is to investigate
the gap between primitivity and quasiprimitivity. We investigate the
pre-primitivity of various classes of transitive groups including groups with
regular normal subgroups, direct and wreath products, and diagonal groups.
In the course of this investigation, we describe all $G$-invariant partitions
for various classes of permutation groups $G$.
We also look briefly at conditions similarly related to other pairs of
conditions, including transitivity and quasiprimitivity, $k$-homogeneity and
$k$-transitivity, and primitivity and synchronization.

\end{abstract}

\section{Introduction}

In his pioneering work on permutation groups in his Second Memoir~\cite{galois},
Galois introduced the notion of primitivity, which has occupied
the attention of mathematicians ever since. However, Neumann~\cite{pmn:galois}
has pointed out that Galois confused two inequivalent conditions for the
transitive permutation group~$G$ on $\Omega$:
\begin{itemize}
\item $G$ preserves no non-trivial partition of $\Omega$ (the trivial
partitions being the partition into singletons and the partition with a
single part);
\item every non-trivial normal subgroup of $G$ is transitive.
\end{itemize}
The first of these conditions is what is now called \emph{primitivity},
while the second is \emph{quasiprimitivity}. Since the orbit partition of
a normal subgroup is $G$-invariant, we see that a primitive group is
quasiprimitive; but the converse is false, as we may see by considering the
regular representation of a non-abelian simple group.

In order to investigate the gap between these two properties, we make the
following definition: The transitive permutation group $G$ on $\Omega$ is
\emph{pre-primitive} if every $G$-invariant partition is the orbit partition
of a subgroup of $G$. We can assume that this subgroup is normal:

\begin{prop}
If a $G$-invariant partition is the orbit partition of a subgroup of $G$, then
it is the orbit partition of a normal subgroup.
\label{p:normal}
\end{prop}

\begin{proof}
The set of permutations fixing all parts of a $G$-invariant partition is a
normal subgroup of $G$. \qed
\end{proof}

\begin{thm}
\begin{enumerate}
\item There are permutation groups which are quasiprimitive but not 
pre-primitive, and permutation groups which are pre-primitive but not
quasiprimitive.
\item A permutation group is primitive if and only if it is quasiprimitive
and pre-primitive.
\end{enumerate}
\end{thm}

\begin{proof} We establish the second statement first. We have noted that a
primitive group is quasiprimitive; it is also pre-primitive, since both trivial
partitions are orbit partitions of subgroups (the trivial group and the whole
group respectively). Conversely, suppose that $G$ is pre-primitive and
quasiprimitive, and let $\Pi$ be a $G$-invariant partition. Then
$\Pi$ is the orbit partition of a subgroup $H$ of $G$. As noted after
the definition, we may assume that $H$ is normal in $G$; now quasiprimitivity
shows that $H$ is trivial or transitive, so $\Pi$ is trivial.

We have seen examples of quasiprimitive groups which are not primitive,
and hence not pre-primitive. For the other case, let $G$ be an abelian
group which is not of prime order, acting regularly. Then the $G$-invariant
partitions are the coset partitions of subgroups of $G$, and so $G$ is
pre-primitive but not primitive, hence not quasiprimitive. \qed
\end{proof}

We give here another general property of pre-primitivity, which it shares with
many permutation group properties.

\begin{thm} 
Pre-primitivity is upward-closed; that is, if $G_1$ and $G_2$ are transitive
permutation groups on $\Omega$ with $G_1$ pre-primitive and $G_1\le G_2$,
then $G_2$ is pre-primitive.
\label{lemma-PP}
\end{thm}

\begin{proof}
With these hypotheses, let $\Pi$ be a $G_2$-invariant partition. Then
clearly $\Pi$ is $G_1$-invariant, so it is the orbit partition of a
subgroup $H$ of $G_1$; and we have $H\le G_2$. \qed
\end{proof}

We also give a group-theoretical characterisation of pre-primitivity.

\begin{thm}
Let $G$ be transitive on $\Omega$, and take $\alpha\in\Omega$. Then $G$ is
pre-primitive if and only if every subgroup $H$ containing $G_\alpha$ has
the form $H=NG_\alpha$ for some normal subgroup $N$ of $G$.
\end{thm}

\begin{proof}
We observe that every subgroup $H$ containing $G_\alpha$ is the stabiliser of
the part containing $\alpha$ of some $G$-invariant partition $\Pi$. If $G$
is pre-primitive, then there is a normal subgroup $N$ of $G$ whose orbits
are the parts of $\Pi$; then the $NG_\alpha$-orbit of $\alpha$ is the part
of $\Pi$ containing $\alpha$, and so $NG_\alpha=H$. Conversely, if
$NG_\alpha=H$ for some normal subgroup $N$ of $G$, then the $N$-orbit of
$\alpha$ is equal to the $H$-orbit, and so is a part of $\Pi$; since $N$ is
normal, every $N$-orbit is a part of $\Pi$. Thus $G$ is pre-primitive. \qed 
\end{proof}

In the remainder of the paper, we consider various classes of transitive
groups, including groups with regular normal subgroups, direct and wreath
products of transitive groups, and diagonal groups~\cite{bcps}. We attempt to
determine when these groups are pre-primitive; in some cases we succeed, in
others we obtain necessary and sufficient conditions which are quite close
together. We report the result of computations on the numbers of transitive
groups of small degree which are pre-primitive, quasiprimitive and primitive
respectively, showing that the first two conditions are approximately
statistically independent. In several cases we determine all the $G$-invariant
partitions for certain types of permutation group.

In the last section we consider similar conditions
relating to other pairs of permutation group properties, the second being
stronger than the first:
\begin{itemize}
\item transitivity and quasiprimitivity;
\item $k$-homogeneity and $k$-transitivity;
\item primitivity and synchronization~\cite{acs,pmn:sync}.
\end{itemize}
These all turn out to be of less interest. The first case is trivial. In the
second, the property we are looking for is Neumann's notion of
\emph{generous $(k-1)$-transitivity}~\cite{pmn:generosity}, which has been
much studied, especially for $k=2$. In the third case, we define an
appropriate property which we call \emph{pre-synchronization}, but we
prove that the only transitive group which is pre-synchronizing but not
synchronizing is the Klein group of order~$4$.

\section{Pre-primitive groups of specific types}

In this major section we discuss some familiar types of transitive permutation
groups with a view to deciding when they are pre-primitive.

\subsection{Groups with regular normal subgroups}
\label{s:rns}

To help fix the ideas, we first discuss groups acting regularly. If $G$ acts
regularly on $\Omega$, then the set $\Omega$ is bijective with $G$ and the
given action is isomorphic to the action by right multiplication. If the point
$\alpha\in\Omega$ corresponds to the identity of $G$, then:
\begin{enumerate}
\item a partition of $G$ is $G$-invariant if and only if it is the
right coset partition of a subgroup of $G$;
\item a partition of $G$ is the orbit partition of a subgroup $H$ of $G$ if and
only if it is the left coset partition of $H$.
\end{enumerate}

For (a), suppose that $H$ is the part of the partition containing the
identity. Then for $h_1,h_2\in H$, multiplication by $h_1^{-1}h_2$ maps
$h_1$ to $h_2$, so fixes $H$; thus $1.(h_1^{-1}h_2)\in H$. So $H$ is a
subgroup of $G$. Then, for any $x\in G$, $Hx$ is a part of the partition.
So the claim is proved.

For (b), let $H$ be a subgroup of $G$. Then the orbit of $H$ containing $g$
is the left coset $gH$. So the claim holds.

It follows that the regular group $G$ is pre-primitive if and only if, for
every subgroup $H$, the left and right coset partitions of $H$ coincide, that
is, $H$ is a normal subgroup. We will state this formally as a corollary of
the main result of this section.

\begin{thm}
Let $G$ be a permutation group on $\Omega$ with a regular normal subgroup
$N$. Then $G$ is pre-primitive if and only if every $G_\alpha$-invariant
subgroup $H$ of $N$ is normal in $N$.
\label{t:rns}
\end{thm}

\begin{proof}
Since $N$ is a regular normal subgroup of $G$, we have $G=NG_\alpha$, and 
we can identify $\Omega$ with $N$ in such a way that $N$ acts by right
multiplication and $G_\alpha$ acts by conjugation. Moreover, we can assume
that $\alpha$ is the identity element of $N$ (see~\cite[Theorem~11.2]{wielandt}).

Suppose that $G$ is pre-primitive. Let $H$ be a $G_\alpha$-invariant subgroup
of $N$. The right coset partition $\Pi$ of $H$ is $N$-invariant. It is also
$G_\alpha$-invariant: for if $Hn$ is a right coset of $H$ and $g\in G_\alpha$,
then $(Hn)^g=H^gn^g=Hn'$ for some $n'\in N$. Since $NG_\alpha=G$,
$\Pi$ is $G$-invariant. Since $G$ is pre-primitive, $\Pi$ is the orbit partition
of a normal subgroup $K$ of $G$. Now $K\le G_\alpha(K\cap H)$, so
$\alpha H=\alpha K=\alpha(K\cap N)$. Since $N$ s regular, it follows that
$K\cap N=H$.

Conversely, assume that every $G_\alpha$-invariant subgroup of $N$ is
normal in $N$. Choose a $G$-invariant partition $\Pi$. Since $\Pi$ is
$N$-invariant it is the right coset partition of a subgroup $K$ of $N$, which
is also $G_\alpha$-invariant, since $G_\alpha$ fixes the part of $\Pi$
containing the identity. Thus $K$ is normal in $N$. The normal subgroup of
$G$ fixing every part of $\Pi$ contains $K$, and $\Pi$ is the orbit partition
of $N\cap K$. \qed
\end{proof}

This result allows us to deal with some special types of permutation groups.
First we state and prove our earlier result about regular groups.
Recall that a \emph{Dedekind group} is a finite group in which every
subgroup is normal. Dedekind~\cite{dedekind} showed:

\begin{thm}
A finite group $G$ is a Dedekind group if and only if either $G$ is abelian,
or $G\cong Q\times A\times B$, where $Q$ is the quaternion group of order~$8$,
$A$ is an elementary abelian $2$-group, and $B$ is an abelian group of odd
order.
\label{t:dedekind}
\end{thm}

\begin{cor}
The regular action of a finite group $G$ is pre-primitive if and only if $G$
is a Dedekind group.
\label{c:regular}
\end{cor}

\begin{proof}
In this case, Theorem~\ref{t:rns} applies with $N=G$ and $G_\alpha=1$. Thus
every subgroup of $N$ is $G_\alpha$-invariant; so $G$ is pre-primitive if and
only if every subgroup of $N$ is normal. \qed
\end{proof}

The \emph{holomorph} of a group $G$ is the semidirect product of $G$ with
$\Aut(G)$. It acts as a permutation group on $G$, where $G$ acts by right
multplication and $\Aut(G)$ in the natural way.

\begin{cor}
For any finite group $G$, the holomorph of $G$ is pre-primitive.
\end{cor}

\begin{proof}
The group $G$ is a regular normal subgroup of its holomorph. A subgroup $H$ of
$G$ is $\Aut(G)$-invariant if and only if it is characteristic; and a
characteristic subgroup is certainly normal. \qed
\end{proof}

The proof actually shows a stronger result: the semidirect product of $G$ by
its inner automorphism group $\Inn(G)$ is pre-primitive.

Finally, since a transitive abelian group is regular, and a direct product of
regular groups (in its product action) is regular, we have the following:

\begin{cor}\label{cor-direct-abelian}
The direct product of transitive abelian groups in its product action is 
pre-primitive.
\end{cor}

\subsection{Direct products}

Next we consider various product constructions for transitive groups and ask,
is it true that if the factors are pre-primitive, then so is the product?
First, the direct product in its product action.

Let $G$ and $H$ be permutation groups on $\Gamma$ and $\Delta$ respectively.
Then the direct product $G\times H$ acts coordinatewise on $\Gamma\times\Delta$,
by $(\gamma,\delta)(g,h)=(\gamma g,\delta h)$. If $G$ and $H$ are transitive
then $G\times H$ is transitive in this action.

It follows from Corollary~\ref{c:regular} that, if two transitive groups
are pre-primitive, then their direct product (in its product action) is 
pre-primitive if the factors are abelian, but may fail to be pre-primitive in
general. (Take two copies of $Q_8$ acting regularly: $Q_8$ is a
Dedekind group but $Q_8\times Q_8$ is not.) So it is natural to ask what
additional conditions on the factors will guarantee pre-primitivity of the
product. To examine these, we look more closely at partitions invariant under
$G\times H$.

Suppose that $G$ and $H$ act transitively on $\Gamma$ and $\Delta$ respectively,
and let $\Pi$ be a $(G\times H)$-invariant partition of $\Gamma\times\Delta$.
We define two partitions of $\Gamma$ as follows.
\begin{itemize}
\item Let $P$ be a part of $\Pi$. Let $P_0$ be the subset of $\Gamma$ defined
by 
\[P_0=\{\gamma\in\Gamma:(\exists\delta\in\Delta)(\gamma,\delta)\in P\}.\]
We claim that the sets $P_0$ arising in this way are pairwise disjoint. For
suppose that $\gamma\in P_0\cap Q_0$, where $Q_0$ is defined similarly from
another part $Q$ of $\Pi$; suppose that $(\gamma,\delta_1)\in P$ and
$(\gamma,\delta_2)\in Q$. There is an element $h\in H$ mapping $\delta_1$ to
$\delta_2$. Then $(1,h)$ maps $(\gamma,\delta_1)$ to $(\gamma,\delta_2)$, and
hence maps $P$ to $Q$, and $P_0$ to $Q_0$; but this element acts trivially on
$\Gamma$, so $P_0=Q_0$. It follows that the sets $P_0$ arising in this way
form a partition of $\Gamma$, which we call the \emph{$G$-projection partition}.
\item Choose a fixed $\delta\in\Delta$, and consider the intersections of
the parts of $\Pi$ with $\Gamma\times\{\delta\}$. These form a partition of
$\Gamma\times\{\delta\}$ and so, by ignoring the second factor, we obtain a
partition of $\Gamma$ called the \emph{$G$-fibre partition}. Now the action
of the group $\{1\}\times H$ shows that it is independent of the element
$\delta\in\Delta$ chosen.
\end{itemize}
We note that the $G$-projection partition and the $G$-fibre partition are both
$G$-invariant, and the second is a refinement of the first. In a similar way
we get $H$-fibre and $H$-projection partitions of $\Delta$, both $H$-invariant.

For a non-trivial example, consider the group $C_8\times C_8$ acting on
$\Gamma\times\Delta$, where each of $\Gamma$ and $\Delta$ is a copy of the
integers modulo~$8$. Take 
\[P=\{(0,0),(0,4),(4,0),(4,4),(2,2),(2,6),(6,2),(6,6)\}.\]

The images of $P$ under $C_8\times C_8$ form a partition with eight parts;
its projection partition on the first coordinate has two parts consisting of
the even and odd elements, while its fibre partition has four parts consisting
of the cosets of $\{0,4\}$.

The next lemma gives some properties of these partitions. First, some
definitions. 

The partial order of \emph{refinement} is defined on partitions of $\Omega$
by the rule that, for partitions $\Pi$ and $\Sigma$, we have
$\Pi\preceq\Sigma$ (read $\Pi$ refines $\Sigma$) if every part of $\Sigma$
is a union of parts of $\Pi$.

If $\Pi$ and $\Sigma$ are partitions of $\Gamma$ and $\Delta$ respectively,
then their \emph{cartesian product} $\Pi\times\Sigma$ is the partition of
$\Gamma\times\Delta$ whose parts are all cartesian products of a part of $\Pi$
and a part of $\Sigma$.

\begin{lemma} Let $G$ and $H$ be transitive permutation groups on $\Gamma$
and $\Delta$ respectively, and let $\Pi$ be a $G$-invariant partition of
$\Gamma\times\Delta$.
\begin{enumerate}
\item The $G$-projection and $G$-fibre partitions of $\Gamma$ are $G$-invariant,
and the second is a refinement of the first.
\item The number $k$ of parts of the $G$-fibre partition contained in a part of
the $G$-projection partition is equal to the corresponding number for $H$.
\item If $k=1$, then $\Pi$ is the cartesian product of a
$G$-invariant partition of $\Gamma$ and an $H$-invariant partition of $\Delta$.
\item If $k>1$, then the set of $k^2$ parts of $F_G\times F_H$ within a part
of $P_G\times P_H$ has the structure of a Latin square, where the first
and second coordinates define the square grid and the parts of $P$ give the
positions of the letters.
\end{enumerate}
\label{lemma:fibre_proj}
\end{lemma}

\begin{proof}
The first statement is clear from the definition.

For the second, let $P_G$ and $F_G$ be the $G$-projection and $G$-fibre
partitions of $\Gamma$, and $P_H$ and $F_H$ the corresponding partitions of
$\Delta$. We claim first that
\[
F_G\times F_H\preceq P\preceq P_G\times P_H.
\]
The second inequality is clear since, if two elements of $\Gamma\times\Delta$
lie in the same part of $P$, then their projections onto $\Gamma$ lie in
the same part of $P_G$, and similarly for $\Delta$ and $P_H$. For the first
inequality, suppose that $\gamma_1$ and $\gamma_2$ lie in the same part of
$F_G$, and $\delta_1$ and $\delta_2$ in the same part of $F_H$. Then there
exists $\delta\in\Delta$ such that $(\gamma_1,\delta)$ and $(\gamma_2,\delta)$
lie in the same part of $\Pi$; applying an element of $\{1\}\times H$, we can
assume that $\delta=\delta_1$. Similarly, we can assume that
$(\gamma_1,\delta_1)$ and $(\gamma_1,\delta_2)$ belong to the same part of
$\Pi$. Now transitivity gives the result.

Now let $A$ and $B$ be parts of $P_G$ and $P_H$, and choose a part $P$ of $\Pi$
which projects onto $A$ and $B$. Any point of $P$ belongs to $a\times b$ for
some parts $a,b$ of $F_G$ and $F_H$ respectively. This induces a bijection
between the parts of $F_G$ in $A$ and the parts of $F_H$ in $B$.

The third part is now clear. 

For the final part, note that the parts of $F_G\times F_H$ within $A\times B$
have the structure of a $k\times k$ square grid. Choose a part $P$ of $\Pi$
within $A\times B$. Then $P$ is a union of parts of $F_G\times F_H$. By the
definition of the fibre partition, the first coordinates of pairs in $P$ with
given second coordinate form a single part of $F_G$, so $P$ contains just one
part of $F_G\times F_H$ within any row of the grid; and similarly for columns.
Since every part of $F_G\times F_H$ is contained in a unique part of $P$, the
result is proved.
\qed
\end{proof}

\begin{thm}\label{t:prim_x_prim}
Let $G \leq \Sym(\Gamma)$ and $H \leq \Sym(\Delta)$ be transitive and let $G\times H$ act component-wise on $\Gamma \times \Delta$. If both $G$ and $H$ are primitive, then $G\times H$ is pre-primitive. 
\end{thm}

\begin{proof}
Since both $G$ and $H$ are primitive, it follows that given a $G\times H$-invariant partition $\Pi$ of $\Gamma\times \Delta$ the $G$ and $H$-fibre partitions and the $G$ and $H$-projection partitions must be trivial. If $|\Gamma| \neq |\Delta|$, then $\Pi$ is one of four possibilities: $\{\Gamma \times \Delta\}$, (the partition with a single part), $\{\{(\gamma, \delta)\} \mid \gamma \in \Gamma, \delta\in \Delta\}$, (the partition into singletons), $\{\Gamma \times \{\delta\} \mid \delta\in \Delta\}$, and $\{\{\gamma\} \times \Delta \mid \gamma \in \Gamma\}$.
Then $\Pi$ is the orbit partition of $G\times H$, $1$, $G\times 1$, $1\times H$ respectively, and hence $G\times H$ is pre-primitive. 

If $|\Gamma| = |\Delta|$, then there is one additional case to consider, namely the one where the $G$ and $H$-fibre partitions are the singletons, and the $G$ and $H$-projection partitions consist of a single part. 
Each part of $P$ induces a bijection between $\Gamma$ and $\Delta$, so the
stabiliser of a point $(\gamma,\delta)$ fixes every point in the part of $P$
containing $(\gamma,\delta)$, and hence is the identity. So $G\times H$ is
regular, whence also $G$ and $H$ are regular. Since they are also primitive,
they are cyclic of prime order, whence $G\times H$ is abelian, and hence by
Corollary~\ref{cor-direct-abelian} it is pre-primitive. \qed
\end{proof}

\begin{thm}
Let $G \leq \Sym(\Gamma)$ and $H \leq \Sym(\Delta)$ be transitive and let $G\times H$ act component-wise on $\Gamma \times \Delta$. If $G$ and $H$ are pre-primitive and the sizes of $\Gamma$ and $\Delta$ are coprime, then $G\times H$ is pre-primitive.
\end{thm}

\begin{proof}
Let $\Pi$ be a $G\times H$-invariant partition. We denote the $G$ and $H$-fibre partitions by $\Sigma_{\Gamma}$ and $\Sigma_{\Delta}$ respectively, and we let $P_G$ and $P_H$ be the $G$- and $H$-projection partitions respectively.

Let $k$ be the number of parts of the $G$-fibre partition in a part of $P_G$.
Then $k$ divides $|\Gamma|$, since the product of the size of a part of the
$G$-fibre partition times $k$ times the number of parts of $P_G$ is equal
to $|\Gamma|$. Similarly $k$ divides $\Delta$. Thus $\Pi=P_G\times P_H$,
by Lemma~\ref{lemma:fibre_proj}.
Since $P_G$ and $P_H$ are orbit partitions of subgroups $G^*$ and $H^*$ of
$G$ and $H$ respectively, $\Pi$ is the orbit partition of $G^*\times H^*$.
\qed
\end{proof}

\begin{cor}
Let $G \leq \Sym(\Gamma)$, $H\leq \Sym(\Delta)$ be transitive and regular. The direct product $G\times H$ acting on $\Gamma \times \Delta$ componentwise is pre-primitive if and only if either $G$ and $H$ are abelian, or one is a
non-abelian Dedekind group and the other an abelian group whose exponent is not
divisible by $4$.
\end{cor}

\begin{proof}
If $G$ and $H$ both act regulary on $\Gamma$ and $\Delta$ respectively, then $G\times H$ acts regularly on $\Gamma \times \Delta$. Therefore, $G\times H$ is pre-primitive if and only if it is Dedekind, which happens only in the two cases
stated, by Dedekind's theorem (Theorem~\ref{t:dedekind}).  \qed
\end{proof}

\begin{thm}
Let $G \leq \Sym(\Gamma)$ and $H\leq \Sym(\Delta)$ be pre-primitive. Suppose
that every $G\times H$-invariant partition $\Pi$ of $\Gamma \times \Delta$ is of
one of the following types:
\begin{itemize}
\item The $G$-fibre partition induced by $\Pi$ on $\Gamma$ is the same as the $G$-projection partition and the $H$-fibre partition induced by $\Pi$ on $\Delta$ is the same as the $H$-projection partition;
\item The $G$ and $H$-projection partitions induced by $\Pi$ on $\Gamma$ and $\Delta$ respectively are the partitions with a single part.
\end{itemize}
then $G\times H$ in its product action is pre-primitive.
\end{thm}

\begin{proof}
If $\Pi$ is of the first form, then by Lemma~\ref{lemma:fibre_proj},
$\Pi = P_G\times P_H$.
Since $G$ is pre-primitive, there exists a subgroup $G^*$ of $G$ whose orbit partition is $P_G$, and similarly there exists some $H^*\leq H$ whose orbit partition is $P_H$. Now it is easy to see that the orbit partition of $G^*\times H^*$ is $\Pi$.

Now suppose that $\Pi$ is of the second type and let $F_G$ and $F_H$ be the $G$- and $H$-fibre partition induced by $\Pi$ on $\Gamma$ and $\Delta$ respectively.
As in the first part, there are subgroups $G^*$ and $H^*$ of $G$ and $H$
respectively (which can be chosen to be normal, by Proposition~\ref{p:normal})
whose orbit partitions are $F_G$ and $F_H$ respectively; and $F_G\times F_H$
is the orbit partition of $G^*\times H^*$. Moreover, we can take $G^*$ to
consist of all elements of $G$ fixing the parts of $F_G$, and similarly for 
$H^*$.

The group $(G\times H)/(G^*\times H^*)$ permutes (faithfully, by the above
remark) the parts of $F_G\times F_H$.
The argument in the second part of Theorem~\ref{t:prim_x_prim} shows that
$G/G^*$ and $H/H^*$ are isomorphic and regular. If $G/G^*$ has a non-trivial
proper subgroup $L/G^*$, and $M/H^*$ is the corresponding subgroup of $H/H^*$,
then the inverse image in $G\times H$ of a diagonal subgroup of 
$L/G^*\times M/H^*$ defines a $G$-invariant partition not of the form in the
theorem. So $G/G^*$ and $H/H^*$ are of prime order $p$. Then $\Pi$ is the
orbit partition of a subgroup of $G\times H$ whose projection onto 
$(G\times H)/(G^*\times H^*)$ is a diagonal subgroup of $C_p\times C_p$.
\qed
\end{proof}

In the next section, we will show a clean converse of these results: if
$G\times H$ is pre-primitive, then both $G$ and $H$ are pre-primitive.

\subsection{Wreath products, imprimitive action}

By contrast with direct products, wreath products are better behaved.
Again let $G$ and $H$ be permutation groups on $\Gamma$ and $\Delta$
respectively. Take $\Omega=\Gamma\times\Delta$, regarded as the disjoint
union of copies $\Gamma_\delta$ of $\Gamma$ indexed by $\Delta$, where
$\Gamma_\delta=\{(\gamma,\delta):\gamma\in\Gamma\}$. The partition of $\Omega$
into the sets $\Gamma_\delta$ will be called the \emph{canonical partition}.
Now $G\wr H$ is generated by 
\begin{itemize}
\item the \emph{base group}, the direct product of $|\Delta|$ copies of $G$
indexed by $\Delta$, where copy $G_\delta$ with index $\delta$ acts on
$\Gamma_\delta$ as $G$ acts on $\Gamma$ and fixes all the other parts of the
canonical partition pointwise;
\item the \emph{top group}, a copy of $H$ acting on the second coordinate of
points in $\Gamma\times\Delta$.
\end{itemize}

For further use, we note a property of this action.

\begin{lemma}
Let $G$ and $H$ be permutation groups on $\Gamma$ and $\Delta$ respectively.
Then the direct product $G\times H$ in its product action is a subgroup of the
wreath product $G\wr H$ in its imprimitive action.
\label{l:dpinwp}
\end{lemma}

\begin{proof}
The top group is isomorphic to $H$ and acts on the second coordinate. If $D$
is the diagonal subgroup of the base group, consisting of elements with all
coordinates equal, then $D$ is isomorphic to $G$ and acts on the first
coordinate. Together these subgroups generate $G\times H$ in its product action.
\qed
\end{proof}

\begin{prop}
Let $G$ and $H$ be transitive permutation groups on $\Gamma$ and $\Delta$
respectively and let $G\wr H$ have its imprimitive action on 
$\Omega=\Gamma\times\Delta$, with canonical partition $\Pi$. If $\Sigma$ is
any $(G\wr H)$-invariant partition, then either $\Sigma\preceq\Pi$ or
$\Pi\preceq\Sigma$.
\end{prop}

\begin{proof} 
Suppose not, and let $A$ be a part of $\Sigma$. Then $A$ intersects two parts
$\Gamma_\delta$ and $\Gamma_{\delta'}$ of $\Pi$ but contains neither. Now
$\Gamma_\delta\cap A$ is a part of a $G$-invariant partition of $\Gamma_\delta$;
by transitivity, we can find an element $g\in G_\delta$ which maps this set
to a disjoint subset of $\Gamma_\delta$. Then the element of the base group
which acts as $g$ on $\Gamma_\delta$ and the identity outside maps $A$ to
a set $A'$ which is neither equal to nor disjoint from $A$, contradicting the
assumption that $\Sigma$ is $(G\wr H)$-invariant. \qed
\end{proof}

From this we see that, apart from the canonical partition, there are just two
types of non-trivial $(G\wr H)$-invariant partitions $\Sigma$:
\begin{itemize}
\item If $\Sigma\preceq\Pi$, then we take a $G$-invariant partition $\Sigma_0$
of $\Gamma$, and copy it to all parts of the canonical partition using the top
group. Since $\Sigma$ is invariant under the top group, the partitions of the
parts of $\Pi$ must correspond in this way.
\item If $\Pi\preceq\Sigma$, then we take an $H$-invariant partition $\Sigma^0$
of $\Delta$, and replace each part $A$ by the union of the sets $\Gamma_\delta$
for $\delta\in A$. For the sets of indices $\delta$ for which $\Gamma_\delta$
is contained in each part of the partition must form an $H$-invariant partition
of $\Delta$.
\end{itemize}

From this, we can prove our main result about the imprimitive action of the
wreath product:

\begin{thm}
Let $G, H$ be transitive groups on $\Gamma$ and $\Delta$ respectively. Then
the wreath product $G\wr H$ in its imprimitive action on $\Gamma\times\Delta$
is pre-primitive if and only if both $G$ and $H$ are pre-primitive.
\label{t:wreath}
\end{thm}

\begin{proof}
Suppose that $G$ and $H$ are pre-primitive, and let $\Sigma$ be a
$G\wr H$-invariant partition of $\Gamma\times\Delta$, different from the
canonical partition $\Pi$.
\begin{itemize}
\item If $\Sigma\preceq\Pi$, then $\Sigma$ is obtained by copying a
$G$-invariant partition $\Sigma_0$ of $\Gamma$ onto each of the parts
$\Gamma_\delta$. By assumption. $\Sigma_0$ is the orbit partition of a subgroup
$N$ of $G$. Then clearly $\Sigma$ is the orbit partition of $N^m$, where
$m=|\Delta|$, since the $\delta$ coordinate of the direct product acts on
$\Gamma_\delta$ with orbit partition $\Sigma_\delta$.
\item If $\Pi\preceq\Sigma$, then $\Sigma$ is obtained by taking the unions
of the sets $\Gamma_\delta$ corresponding the the points $\delta$ in each
part of a $H$-invariant partition $\Sigma^0$ of $\Delta$. By assumption, the
parts of $\Sigma^0$ are the orbits of a subgroup $N$ of $H$. Then clearly the
parts of $\Sigma$ are the orbits of the subgroup $G^m.N=G\wr N$ of $G\wr H$.
\end{itemize}

Now suppose conversely that $G\wr H$ is pre-primitive.
\begin{itemize}
\item Let $\Sigma_0$ be a $G$-invariant partition of $\Gamma$. Then the
partition $\Sigma_0$ obtained by copying $\Sigma$ onto each part of the
canonical partition $\Pi$ is $G\wr H$-invariant, and so is the orbit partition
of a subgroup of $G\wr H$. The group of all elements fixing this partition
induces $H$ on the parts of the canonical partition, and has the form
$N\wr H$, where $N$ is the stabiliser of all the parts of $\Sigma_0$. So
$\Sigma_0$ is the orbit partition of the subgroup $N$ of $G$. Thus, $G$ is
pre-primitive.
\item Let $\Sigma^0$ be a $H$-invariant partition of $\Delta$. The partition
$\Sigma$ whose parts are unions of parts of $\Pi$ indexed by elements of a part
of $\Sigma^0$ is $(G\wr H)$-invariant, and so is the orbit partition of a
subgroup of $G\wr H$. Clearly this subgroup has the form $G\wr N$, where $N$
is a subgroup of $H$ whose orbit partition is $\Sigma^0$. So $H$ is
pre-primitive. \qed
\end{itemize}
\end{proof}

Here is the promised result for direct products:

\begin{thm}
Let $G\leq \Sym(\Gamma)$ and $H\leq \Sym(\Delta)$ be transitive groups. If the
direct product $G\times H$ in its product action is pre-primitive then both $G$ and $H$ are pre-primitive.
\end{thm}

\begin{proof}
Suppose that $G\times H$ is pre-primitive. Since the direct product is
embedded in the wreath product in its imprimitive action (Lemma~\ref{l:dpinwp}),
we have that $G\wr H$ in its imprimitive action is pre-primitive by
Proposition~\ref{lemma-PP}; so $G$ and $H$ are pre-primitive, by
Theorem~\ref{t:wreath}. \qed
\end{proof}

\subsection{Wreath products, product action}

This is the same group (up to isomorphism) but a different action. As before
let $G$ act on $\Gamma$ and $H$ on $\Delta$. Then $G\wr H$ acts on the
Cartesian product of $|\Delta|$ copies of $\Gamma$, which we regard as the
set of words of length $N=|\Delta|$ over the alphabet $\Gamma$. The factor
$G_\delta$ of the base group acts on the symbols in position $\delta$, fixing
the symbols in the other positions; the top group acts by permuting the
coordinates.

In this section we examine the product action of the wreath product of two
permutation groups $G$ and $H$. First some preliminary remarks.

To begin, $G\wr H$ is transitive if and only if $G$ is transitive, independent
of $H$ (since then the base group $G^n$ is transitive in the product action).
It is known that $G\wr H$ is primitive if and only if $G$ is primitive but
not cyclic of prime order and $H$ is transitive; a similar result holds for
quasiprimitivity~\cite[Theorem 5.8]{ps}. Apart from the first example below,
we assume that $H$ is transitive in this section.

\paragraph{Remark} If $G$ is transitive and abelian, then $G\wr H$ is
pre-primitive, since the base group is transitive and abelian.

\paragraph{Remark} The quaternion group $Q$ is pre-primitive, but
$G=Q\wr C_2$ is not. For $G$ has a regular normal subgroup $Q^2$; the
stabiliser $G_\alpha$ interchanges the two factors. So, if $a$ and $b$
generate $Q$, then the subgroup generated by $(a,a)$ is $G_\alpha$-invariant,
but is not normal, since $b^{-1}ab=a^{-1}$, and so $(1,b)$ conjugates $(a,a)$
to $(a,a^{-1})$, which is not in the subgroup generated by $(a,a)$.
By Theorem~\ref{t:rns}, $G$ is not pre-primitive.

\medskip

However, with an extra assumption on $G$, we do get pre-primitivity of
$G\wr H$. Let $n$ be the degree of $H$, and $\Gamma$ the set on which $G$
acts. Given a partition $\Sigma$ of $\Gamma$, we denote by $\Sigma^n$ the
partition of $\Gamma^n$ in which $(a_1,\ldots,a_n)$ and $(b_1,\ldots,b_n)$ 
belong to the same part if and only if $a_i$ and $b_i$ belong to the same
part of $\Sigma$ for $i=1,\ldots,n$.

\begin{thm}
Let $G$ and $H$ be transitive permutation groups on $\Gamma$ and $\Delta$
respectively. Assume that $G$ is pre-primitive and has the property that the
stabiliser of a point fixes no additional points (equivalently, this stabiliser
is equal to its normaliser in $G$). Then $G\wr H$, in the product
action, is pre-primitive.
\end{thm}

\begin{proof}
Let $\Pi$ be a non-trivial $(G\wr H)$-invariant partition of $\Gamma^n$, where
$n=|\Delta|$. We claim that
\begin{quote}
Let $P$ be a part of $\Pi$, and let $i\in\Delta$. Then $P$ contains two
$n$-tuples which agree in all positions except position $i$.
\end{quote}
To see this, choose $(a_1,\ldots,a_n)\in P$. Since $\Pi$ is non-trivial,
we can choose a different element $(b_1,\ldots,b_n)\in P$; since $H$ is
transitive on the coordinates, we can assume that $b_i\ne a_i$. Now consider
the subgroup of the base group which acts as the identity on all positions
different from the $i$-th and acts as $G_{a_i}$ on the $i$-th position. By
assumption, this subgroup fixes $P$ and contains an element mapping $b_i$
to $b_i'\ne b_i$. Then $(b_1,\ldots,b_i,\ldots,b_n)$ and
$(b_1,\ldots,b_i',\ldots,b_n)$ are the required $n$-tuples.

Now $P$ is a block of imprimitivity for $G\wr H$, and so its setwise
stabiliser acts transitively on it. Suppose that for one (and hence all)
elements of $P$, there are $k$ other elements of $P$ differing from the
chosen one only in the $i$-th position. The transitivity of $H$ shows that this
number is independent of $i$. Also, the mapping taking one such tuple to
another can be chosen to lie in the base group and to stabilise $P$. Since
elements of the base group acting on different coordinates commute, we see
that $P$ is the Cartesian product of $k$-subsets of $\Gamma$, one for each
coordinate. We show that the partition $\Pi$ has the form $\Sigma^n$ for some
partition $\Sigma$ of $\Gamma$.

First we look at the parts of $\Pi$ containing diagonal elements
$(a,a,\ldots,a)$ of $\Omega$. Such a part $P(a)$ is invariant under the action
of the top group $H$, which is transitive on the components; so the subsets
of the different copies of $\Gamma$ making up $P(a)$ are all equivalent under
the natural bijections between these copies. Hence the sets
$P(a)\cap\Gamma_\delta$ for $a\in\Gamma$ form a partition of $\Gamma_\delta$,
and all these parts correspond.

Now take a general element $(a_1,\ldots,a_n)$ of $\Omega$. The unique part $P$
containing it can be mapped to $P(a_1)\cap\Gamma_1$ by an element of the
base group fixing $\Gamma_1$ pointwise, so $P\cap\Gamma_1=P(a_1)\cap\Gamma_1$.
This shows that there is a unique partition $\Sigma_\delta$ of each set
$\Gamma_\delta$ made up of the projections of parts of $\Pi$ onto
$\Gamma_\delta$; and these parts correspond under the natual maps between
different sets $\Gamma_\delta$. This shows that $\Pi=\Sigma^n$, as claimed.

Now by assumption, $G$ is pre-primitive, so $\Sigma$ is the orbit partition
of a normal subgroup $N$ of $G$. But now it follows that $\Pi$ is the orbit
partition of the subgroup $N^n$ of the base group of $G\wr H$. Hence
$G\wr H$ is pre-primitive. \qed
\end{proof}

\paragraph{Remark} If $G$ is primitive, the extra condition on $G$ in this
theorem excludes only the cyclic groups of prime order.

\medskip

In the other direction, things are simpler:

\begin{thm}
If $G\wr H$ in the product action is pre-primitive then $G$ is pre-primitive.
\end{thm}

\begin{proof} Let $\Sigma$ be a $G$-invariant partition of $\Gamma$. It is
easy to see that $\Sigma^n$ is a $(G\wr H)$-invariant partition of $\Gamma^n$.
By hypothesis, it is the orbit partition of a normal subgroup $N$ of $G\wr H$.
This subgroup clearly contains the top group $H$, and its intersection with
the base group induces on the first coordinate a subgroup of $G$ whose orbit
partition is $\Sigma$. \qed
\end{proof}

\subsection{Diagonal groups}

Diagonal groups arose in the celebrated O'Nan--Scott Theorem. However, they
form a much larger class of transitive groups, discussed in detail in
\cite{bcps}. Let $m$ be a positive integer and $T$ a group. Then the
\emph{diagonal group} $D(T,m)$ is the group of permutations on $\Omega=T^m$
(where we distinguish $m$-tuples in $\Omega$ by putting them in square brackets)
generated by the following elements:
\begin{enumerate}
\item elements of $T^m$, acting by right multiplication;
\item elements of $T$, acting simultaneously on all coordinates by left
multiplication (that is, $t\in T$ maps $[x_1,\ldots,x_m]$ to
$[t^{-1}x_1,\ldots,t^{-1}x_m]$);
\item automorphisms of $T$, acting simultaneously on all coordinates;
\item elements of the symmetric group $S_m$, permuting the coordinates;
\item the map
\[[t_1,t_2,\ldots,t_m]\mapsto[t_1^{-1},t_1^{-1}t_2,\ldots,t_1^{-1}t_m].\]
\end{enumerate}
Note that the permutations of types (d) and (e) generate a group isomorphic
to $S_{m+1}$.

We are going to find a sufficient condition for $D(T,m)$ to be pre-primitive.

Permutations of type (a) constitute a regular subgroup. Types (c), (d) and (e)
generate the stabiliser of a point. Type (b) are not actually necessary, since
any two of left multiplication, right multiplication and conjugation by a
diagonal element generate the third. The regular subgroup $T^m$ is not normal
if $T$ is nonabelian, so the results of Section~\ref{s:rns} do not immediately
apply; but we will see that very similar results hold.

\begin{prop}
Let $G=D(T,m)$.
\begin{enumerate}
\item A $G$-invariant partition is the right coset partition of a subgroup
$H$ of $T^m$ normalised by the elements of types (c), (d) and (e).
\item If any subgroup $H$ of $T^m$ which is normalised by elements of types
(c), (d) and (e) is normal in $T^m$, then $G$ is pre-primitive.
\end{enumerate}
\label{p:diag}
\end{prop}

\begin{proof}
(a) Let $\Pi$ be a $G$-invariant partition, and let $H$ be the part of $\Pi$
containing the identity. Since $T^m$ is regular, the arguments of 
Section~\ref{s:rns} show that $H$ is a subgroup of $T^m$ and $\Pi$ is its
right coset partition. 

If $\Pi$ is invariant under the point stabiliser, then the same is true for
$H$. In other words, if $(h_1,\ldots,h_m)\in H$, then 
\begin{itemize}
\item $(h_1^\alpha,\ldots,h_m^\alpha)\in H$ for all $\alpha\in\Aut(T)$;
\item $(h_{1\pi},\ldots,h_{m\pi})\in H$ for all $\pi\in S_m$;
\item $(h_1^{-1},h_1^{-1}h_2,\ldots,h_1^{-1}h_m)\in H$.
\end{itemize}

Conversely, suppose that $H$ is invariant under these transformations.
Take any right coset of $H$, say $H(g_1,\ldots,g_m)$.
\begin{itemize}
\item Applying an automorphism $\alpha$ of $T$, we find
\[(h_1g_1,\ldots,h_mg_m)^\alpha=(h_1^\alpha,\ldots h_m^\alpha)(g_1^\alpha,
\ldots,g_m^\alpha)\in H(g_1^\alpha,\ldots,g_m^\alpha)\]
since $H$ is invariant under the coordinatewise action of $\alpha$.
\item Applying a permutation $\pi$ to the subscripts, we find
\[(h_1g_1,\ldots,h_mg_m)^\pi=(h_{1\pi},\ldots,h_{m\pi})
(g_{1\pi},\ldots,g_{m\pi})\in H(g_{1\pi},\ldots,g_{m\pi})\]
since $H$ is invariant under $\pi$ applied to the subscripts;
\item Applying the map $\epsilon$ of type (e), we find
\begin{eqnarray*}
(h_1g_1,\ldots,h_mg_m)\epsilon &=&
(g_1^{-1}h_1^{-1},g_1^{-1}h_1^{-1}h_2g_2,\ldots,g_1^{-1}h_1^{-1}h_mg_m) \\
&=& ((h_1^{-1})^{g_1}g_1^{-1},(h_1^{-1}h_2)^{g_1}g_1^{-1}g_2,\ldots,
(h_1^{-1}h_m)^{g_1}{g_1^{-1}g_m} \\
&\in& H(g_1,\ldots,g_m)\epsilon
\end{eqnarray*}
since $H$ is invariant under both conjugation by $g_1$ and $\epsilon$.
\end{itemize}
So $\Pi$ is invariant under these three types of element.

\medskip

(b) If $H$ is a normal subgroup of $T^m$, then its right coset partition
coincides with its orbit partition. \qed 
\end{proof}

\begin{thm}
Let $T$ be a finite group and $m$ a positive integer. Suppose that the following
property holds:
\begin{quote}
If $K$ is any characteristic subgroup of $T$, and $L$ the subgroup of $K$ 
generated by the $(m+1)$st powers and commutators of elements in $K$, then
every subgroup of $K$ containing $L$ is normal in $T$.
\end{quote}
Then $G=D(T,m)$ is pre-primitive.
\label{t:diag}
\end{thm}

\begin{proof}
We examine further the right coset partition of a subgroup $H$ of $T^m$
invariant under the transformations (c), (d) and (e). Let $\pi_i$ be the
projection of $H$ onto the $i$th direct factor. Since $H$ is invariant
under permutations, $\pi_i(H)=K$ (as subgroup of $T$) is independent of the
index $i$; since $H$ is invariant under automorphisms, $K$ is a characteristic
subgroup of $T$. Thus, $H\le K^m$.

Take $(g_1, g_2, \ldots, g_m)\in H$. Then
$(g_1^{-1}, g_1^{-1}g_2, \ldots, g_1^{-1}g_m)\in H$. 
It follows by closure that $(g_1^{-2}, g_1^{-1}, \ldots, g_1^{-1})\in H$,
whence $(g_1^2,g_1,\ldots,g_1)\in H$. The same is true with the first and
second coordinates swapped; so $(g_1,g_1^{-1},1,\ldots,1)\in H$.

From this we make two deductions:
\begin{itemize}
\item For any $g_1,g_2\in H$, the elements $(g_1,g_1^{-1},1,\ldots,1)$,
$(g_2,g_2^{-1},1,\ldots,1)$ and $(g_1g_2,(g_1g_2)^{-1},1,\ldots,1)$ all
belong to $H$. Multiplying the first two by the inverse of the third and
swapping the first two coordinates, we find that $([g_1,g_2],1,\ldots,1)\in H$.
\item We know $(g_1^2,g_1,\ldots,g_1)\in H$. Multiplying successively by the
elements with $g_1$ in the first coordinate and $g_1^{-1}$ in the $i$th, for
$i=2,\ldots,m$, we get $(g_1^{m+1},1,\ldots,1)\in H$.
\end{itemize}

Since this holds with any coordinate replacing the first, we see that
$L^m\le H$, where $L$ is the subgroup of $K$ generated by $(m+1)$st powers
and commutators. Thus $H/L^m$ is contained in $(K/L)^m$. Note that $K/L$
is an abelian group $A$ of exponent dividing $m+1$.

Let $M$ be the subgroup of $A^m$ consisting of $m$-tuples for which the
product of the coordinates is $1$. Then $M$ is generated by elements having
one coordinate $a\in K/L$, one coordinate $a^{-1}$, and the remaining
coordinates $1$. Since $(g,g^{-1},1,\ldots,1)\in H$ for
all $g\in K$, we see that $H$ contains the inverse image of $M$ in $K^m$.

Now any subgroup of $(K/L)^m$ containing $A$ must have the form
\[\{(a_1,\ldots,a_m):a_1\cdots a_m\in B\}\]
where $B$ is a subgroup of $A$; and $H$ is normal in $T^m$ if and only if
the inverse image of $B$ in $T$ is normal in $T$. So, with our hypothesis,
$H$ is normal in $T^m$, and $G$ is pre-primitive, by
Proposition~\ref{p:diag}(b). \qed 
\end{proof}

There are several simpler conditions which guarantee that the hypotheses of
this theorem are satisfied.

\begin{cor}
Suppose that one of the following holds:
\begin{enumerate}
\item $|T|$ is coprime to $m+1$;
\item $T$ is supersoluble;
\item $T$ is a direct product of non-abelian simple groups.
\end{enumerate}
Then $D(T,m)$ is pre-primitive.
\end{cor}

\begin{proof}
In cases (a) and (c), the subgroup $L$ of $K$ is equal to $K$, since there is
no nontrivial abelian quotient with exponent dividing $m+1$. 

Suppose that $T$ is supersoluble, $K$ is characteristic in $T$, $L$ the
subgroup of $K$ generated by $(m+1)$st powers and commutators, and $M$ a
subgroup of $K$ containing $L$ which is not normal in $T$. Then there is
an element $t\in T$ which does not normalise $M$, and so an element
$u\in M/L$ such that $\langle u\rangle$ is not normalised by $t$. But this
means that there is a chief factor of $T$ in $M/L$ containing $u$ and $u^t$
which is not cyclic, contradicting the fact that $T$ is supersoluble. \qed 
\end{proof}

It may be that the converse of Theorem~\ref{t:diag} is true; we have not been
able to decide this.

\paragraph{Example} It can be verified by using GAP~\cite{GAP} that
$G=D(A_4,3)$ is not pre-primitive. In this case, with $K=V_4$, $L=1$, we
see that there are subgroups of $K$ which are not normal in $T=A_4$.

\subsection{Groups with pre-primitive subgroups}

According to \emph{Jordan's theorem}, a primitive group containing a transitive
subgroup on a subset $\Delta$ of $\Omega$, fixing the points outside $\Delta$,
is $2$-transitive. We will prove a somewhat similar theorem for pre-primitivity.

\begin{thm}
Suppose that $G$ is a transitive permutation group on $\Omega$. Suppose that
$\Delta$ is a subset of $\Omega$ satisfying $|\Delta|>|\Omega|/2$ and $H$ a
subgroup of $G$ which fixes every point outside $\Delta$ and acts 
pre-primitively on $\Delta$. Then $G$ is pre-primitive.
\end{thm}

\begin{proof}
Suppose that $\Pi$ is a $G$-invariant partition. No part of $\Pi$ can intersect
both $\Delta$ and $\Omega\setminus\Delta$ non-trivially, since such a part
would be fixed by $G$ and therefore would contain the whole of $\Delta$ (by
transitivity of $H$). So every part of $\Pi$ is either contained in or
disjoint from $\Delta$.

The parts contained in $\Delta$ form a $H$-invariant partition of $\Delta$,
and so by hypothesis form the orbit partition of a subgroup $K$ of $H$. Now
every conjugate of $K$ in $G$ fixes all parts of $\Pi$, and so they generate
a group whose orbit partition is $\Pi$. 

Since this holds for every $G$-invariant partition, $G$ is pre-primitive. \qed 
\end{proof}

\section{Data on small transitive groups}

We remarked earlier that pre-primitivity and quasiprimitivity are logically
independent. We might ask whether they are statistically independent, in the
sense that if we pick an isomorphism type of transitive permutation group of
degree $n$ at random, the events that it is pre-primitive and that it is
quasiprimitive are uncorrelated. If $T(n)$, $P(n)$, $QP(n)$ and $PP(n)$
denote the numbers of transitive, primitive, quasiprimitive and pre-primitive
groups of degree $n$ up to permutation isomorphism, this is equivalent to
asking whether $T(n)P(n)=QP(n)PP(n)$. This equation is true in some cases (for
example, if $n$ is prime, then every transitive group of degree $n$ is
primitive, so the four numbers are equal). In general, it seems to be roughly
true. Table~\ref{tpqppp}, computed from the library in GAP~\cite{GAP}, gives
the values of the four functions with $10\le n\le 20$, and the correlation
coefficient of the properties ``pre-primitve'' and ``quasiprimitiive'' when
a transitive group of degree $n$ is chosen uniformly at random.

\begin{table}[htbp]
\[\begin{array}{c|cccc|c}
n & T(n) & P(n) & PP(n) & QP(n) & \hbox{correlation} \\
\hline
10 & 45 & 9 & 42 & 9 & 0.0133 \\
11 & 8 & 8 & 8 & 8 & 0 \\
12 & 301 & 6 & 276 & 7 & 0.0014 \\
13 & 9 & 9 & 9 & 9 & 0 \\
14 & 63 & 4 & 59 & 5 & -0.0108 \\
15 & 104 & 6 & 102 & 8 & -0.0178 \\
16 & 1954 & 22 & 1833 & 22 & 0.0007 \\
17 & 10 & 10 & 10 & 10 & 0 \\
18 & 983 & 4 & 900 & 4 & 0.0003 \\
19 & 8 & 8 & 8 & 8 & 0 \\
20 & 1117 & 4 &  1019 & 10 & -0.0046
\end{array}\]
\caption{\label{tpqppp}Numbers of transitive groups, etc.}
\end{table}

It might be useful to have a bound on the correlation coefficient or some
evidence about its sign, but we have been unable to do this. The data suggest
that most transitive groups are pre-primitive and most quasiprimitive groups
are primitive.

\section{Degrees for which all transitive groups are pre-primitive}

Let
\[\mathcal{S}=\{n\in\mathbb{N}:\hbox{every transitive group of degree $n$
is pre-primitive}\}.\]

\paragraph{Problem} Describe the set $\mathcal{S}$.

\medskip

We give some context and then give upper and lower bounds for this set.

One could ask in a similar way about the set of natural numbers $n$ for which
every pre-primitive group of degree $n$ is primitive. But this is easily
seen to be just the set of prime numbers. For, if $n$ is composite, say $n=ab$,
then $S_a\wr S_b$ in its imprimitive action is pre-primitive.

What about the set of natural numbers $n$ for which the only primitive groups
of degree $n$ are the symmetric and alternating groups? This question has a
longer history: for example Mathieu thought about it. But as one of the first
applications of the Classification of Finite Simple Groups, the authors of
\cite{cnt} showed that this set contains almost all natural numbers. More
precisely, if $E$ is the complementary set (for which non-trivial primitive
groups exist), then they showed that
\[|E\cap\{1,\ldots,n\}|=2\pi(n)+(1+\sqrt{2})n^{1/2}+O(n^{1/2}/\log n),\]
where $\pi(n)$ is the number of primes in $\{1,\ldots,n\}$.

\medskip

Clearly our set $\mathcal{S}$ contains all prime numbers, since a transitive
group of prime degree is primitive. It also contains all squares of primes:

\begin{prop}
A transitive permutation group of degree $p^2$, where $p$ is prime, is
pre-primitive.
\end{prop}

\begin{proof}
It suffices to prove the result in the case
where the group $G$ is a $p$-group. For the Sylow $p$-subgroup of a
transitive group $G$ of prime power degree is transitive, and if it is
pre-primitive then so is $G$. So suppose that $G$ is a $p$-group.

Let $\Pi$ be any $G$-invariant partition; it consists of $p$ sets of size $p$.
Let $N$ be the subgroup of $G$ fixing a part of $\Pi$. Then $|G:N|=p$ and
$N$ fixes all the blocks. But $N$ cannot fix a point, since a point stabiliser
has index $p^2$. So the blocks are orbits of $N$. Thus $G$ is pre-primitive.
\qed
\end{proof}

In the other direction, let $\mathcal{A}$ be the set of natural numbers $n$
for which every group of order $n$ is abelian. This well-studied set consists
of all numbers which are not divisible by $p^3$ (where $p$ is prime), or by
$pq$ (where $p$ and $q$ are primes with $q\mid p-1$) or by $p^2q$ (where $p$
and $q$ are primes with $q\mid p+1$). (This result is ``folklore'', but we
have been unable to find a good reference.)

\begin{prop}
$\mathcal{S}\subseteq\mathcal{A}$.
\end{prop}

\begin{proof}
A little thought shows that $\mathcal{A}$ is also the set of natural
numbers $n$ for which every group of order $n$ is Dedekind; and, if $n$ is
not in this set, then a non-Dedekind group acting regularly is not
pre-primitive, so $n\notin\mathcal{S}$. \qed
\end{proof}

Strict inequality holds in both cases:
\begin{itemize}
\item The groups $A_5$ and $S_5$ have transitive imprimitive actions on $15$
points (on the cosets of a Sylow $2$-subgroup). Since $A_5$ is simple, this
action of $A_5$ is quasiprimitive, and so not pre-primitive; the same is true
for~$S_5$. So $15\in\mathcal{A}\setminus\mathcal{S}$.
\item Computation shows that all transitive groups of degrees $33$ and $35$
are pre-primitive. So
$33,35\in\mathcal{S}\setminus(\mathcal{P}\cup\mathcal{P}_2)$, where
$\mathcal{P}$ is the set of primes and $\mathcal{P}_2$ the set of primes 
squared.
\end{itemize}

A special case of the general problem which may be tractable is to find which
products of two distinct primes belong to $\mathcal{S}$. Suppose that $p$ and
$q$ are primes with $q<p$. We have seen that, if $q\mid p-1$, then
$pq\notin\mathcal{S}$. Further examples can be constructed as follows. Let
$p=2^d+1\ge5$ be a Fermat prime, and let $q$ be a prime divisor of $p-2$.
The group $G=\mathrm{PSL}(2,2^d)$ has an imprimitive action on $pq$ points
(the stabiliser being a subgroup of index $q$ in the Sylow $2$-normaliser).
Since this group is simple, the action is quasiprimitive, and so not 
pre-primitive.

\section{Related concepts}

The guiding principle behind the definition of pre-primitivity was to find a
condition logically independent of quasiprimitivity such that its conjunction
with quasiprimitivity is equivalent to primitivity.

We could now consider playing the same game with other pairs of properties
of permutation groups. We give three examples, and invite readers to consider
others. 

\subsection{From transitivity to quasiprimitivity}

We want a property, which we shall call \emph{pre-QP}, which together with
transitivity is equivalent to quasiprimitivity. It is clear that such a 
property can be defined as follows: The permutation group $G$ on $\Omega$
is pre-QP if its action on each of its orbits is quasiprimitive (equivalently,
every normal subgroup of $G$ acts either transitively or trivially on each
$G$-orbit).

This property does not have such a rich theory as pre-primitivity, so we say
no more about it.

\subsection{From $k$-homogeneity to $k$-transitity}

Let $k$ be a positive integer less than $|\Omega|$. A permutation group $G$ on
$\Omega$ is \emph{$k$-homogeneous} if its action on the set of $k$-element
subsets of $\Omega$ is transitive, and is \emph{$k$-transitive} if its action
on the set of ordered $k$-tuples of distinct elements of $\Omega$ is transitive.
Both of these conditions have been intensively studied.

The property that lifts $k$-homogeneity to $k$-transitivity is also known,
having been first given by Neumann~\cite{pmn:generosity}. The permutation group
$G$ is \emph{generously $(k-1)$-transitive} if the setwise stabiliser in $G$
of any $k$-set acts on it as the symmetric group $S_k$. Neumann showed that
this condition implies $(k-1)$-transitivity. In the case $k=2$, it is
equivalent to requiring that all the orbitals of $G$ are self-paired. We have
nothing more to add here.

\subsection{From primitivity to synchronization}

The property of synchronization comes from automata theory by way of
semigroup theory, and is discussed in \cite{acs}. We say that the permutation
group $G$ on $\Omega$ is \emph{synchronizing} if, for any map
$f:\Omega\to\Omega$ which is not a permutation, the monoid
$\langle G,f\rangle$ generated by $G$ and $f$ contains an element of rank~$1$
(that is, one which maps the whole of $\Omega$ onto a single point).

For our purposes, the most useful characterisation of this property is due to
Neumann~\cite{pmn:sync}. We say that a partition $\Pi$ of $\Omega$
is \emph{section-regular} for the permutation group $G$ on $\Omega$ if there
exists a subset $A$ of $\Omega$ such that $Ag$ is a section (transversal) of
$\Pi$ for all $g\in G$. Now Neumann showed that a permutation group is
synchronizing if and only if it has no non-trivial section-regular partition
(where the trivial partitions are as described earlier).

From this it is clear that synchronization implies primitivity: for a
$G$-invariant partition $\Pi$ is section-regular (simply take $A$ to
be any transversal to $\Pi$). The converse is false, as numerous
examples show.

Accordingly, we will say that the transitive permutation group $G$ on $\Omega$
is \emph{pre-synchronizing} if every section-regular partition for $G$ is
$G$-invariant. 

So our interest is in pre-synchronizing groups which are imprimitive. It turns
out that these can be completely classified:

\begin{thm}
Let $G$ be a transitive imprimitive permutation group which is
pre-synchronizing. Then $G$ is isomorphic to the Klein group, in its regular
action of degree~$4$.
\end{thm}

\begin{proof}
Let $\Pi$ be a non-trivial $G$-invariant partition, and $A$ a part of $\Pi$.
Then the images of $A$ under $G$ are the parts of $\Pi$. Take $\Sigma$ to be
any partition such that each part of $\Sigma$ is a transversal for $\Pi$. Then
$Ag$ is a transversal for $\Sigma$, for all $g\in G$; in other words, $\Sigma$
is section-regular.

Since $G$ is pre-synchronizing, it follows that $\Sigma$ is $G$-invariant.
Suppose that $k$ is the size of a part $B$ of $\Sigma$, and let $\Sigma'$ be
another  transversal for $P$ satisfying $|B\cap B'|=k-1$. Then $B\cap B'$ is a
block of imprimitivity contained in $B$. So $k-1$ divides $k$, whence $k=2$.

Now let $m$ be the number of parts of $\Sigma$ (the size of $A$). Since
$\Sigma$ is a non-trivial $G$-invariant partition, we can run the same argument
with $\Pi$ and $\Sigma$ interchanged to conclude that also $m=2$, whence $G$
is a permutation group of degree~$4$. Since it preserves at least two distinct
partitions into two sets of size $2$, we conclude that $G$ must be the Klein
group of order~$4$. \qed 
\end{proof}

\section*{Acknowledgements}

The first author was funded by a StARIS research internship from the University
of St Andrews.

No data were used in the preparation of this paper apart from the GAP transitive
groups library.

The authors declare no conflict of interest.

The authors are grateful to a reviewer for detailed and thoughtful comments
which have materially improved the paper.

\end{document}